\documentclass[a4paper,11pt,reqno]{amsart}
\usepackage{a4wide}

\usepackage{mathdef}

\theoremstyle{plain}
\newtheorem{thm}{Theorem}

\usepackage{tikz}
\usepackage{cite}

\usepackage[draft=false,kerning=true]{microtype}

\newcommand{\vc}{\mathbf}

\renewcommand{\wt}{\mathsf{w}}
\newcommand{\co}{\mathsf{c}}
\newcommand{\la}{\ell}
\newcommand{\di}{d}
\newcommand{\sgn}{\operatorname{sgn}}

\newcommand{\Co}{\mathcal{C}}
\newcommand{\Pa}{\mathcal{P}}
\newcommand{\U}{\mathcal{U}}
\newcommand{\W}{\mathcal{W}}

\newcommand{\Arg}{\operatorname{Arg}}

\begin{document}

% ------------------------------------------------------------------------------

\thispagestyle{plain}

% ------------------------------------------------------------------------------

\title[Compositions into Powers of $b$]{Compositions into Powers of $b$: \\
  Asymptotic Enumeration and Parameters}

% ------------------------------------------------------------------------------

\author{Daniel Krenn}
\address{Daniel Krenn \\
Institute of Analysis and Computational Number Theory (Math A) \\
%Institute of Optimization and Discrete Mathematics (Math B) \\
Graz University of Technology \\
Steyrergasse 30 \\
8010 Graz \\
Austria
}
\email{\href{mailto:math@danielkrenn.at}{math@danielkrenn.at} \textit{or}
  \href{mailto:krenn@math.tugraz.at}{krenn@math.tugraz.at}}

% ------------------------------------------------------------------------------

\author{Stephan Wagner}
\address{Stephan Wagner \\
Department of Mathematical Sciences \\
Stellenbosch University \\
Private Bag X1 \\
Matieland 7602 \\
South Africa
}
\email{\href{mailto:swagner@sun.ac.za}{swagner@sun.ac.za}}

% ------------------------------------------------------------------------------

\thanks{This material is based upon work supported by the National Research
  Foundation of South Africa under grant number 70560.}

\thanks{Daniel Krenn is supported by the Austrian Science Fund (FWF): P24644
  and by the Austrian Science Fund (FWF): W1230, Doctoral Program ``Discrete
  Mathematics''.}

% ------------------------------------------------------------------------------

\thanks{The authors would like to thank Christian Elsholtz for pointing us at the problems discussed in this paper.}

% ------------------------------------------------------------------------------

\date{\today}

% ------------------------------------------------------------------------------

\keywords{compositions, powers of $2$, 
  infinite transfer matrices, asymptotic enumeration}

% ------------------------------------------------------------------------------

\begin{abstract}
  For a fixed integer base $b\geq2$, we consider the number of compositions
  of~$1$ into a given number of powers of~$b$ and, related, the maximum number
  of representations a positive integer can have as an ordered sum of powers
  of~$b$.

  We study the asymptotic growth of those numbers and give precise asymptotic
  formulae for them, thereby improving on earlier results of Molteni. Our
  approach uses generating functions, which we obtain from infinite transfer
  matrices.

  With the same techniques the distribution of the largest denominator and the
  number of distinct parts are investigated.
\end{abstract}

\maketitle

% ------------------------------------------------------------------------------

\section{Introduction}

% ------------------------------------------------------------------------------

Representations of integers as sums of powers of $2$ occur in various contexts, most notably of course in the usual binary representation.
\emph{Partitions} of integers into powers of $2$, i.e., representations of the
form
\begin{equation}\label{eq:partition}
\ell = 2^{a_1} + 2^{a_2} + \cdots + 2^{a_n}
\end{equation}
with nonnegative integers $a_1 \geq a_2 \geq \cdots \geq a_n$ (not necessarily distinct!) are also known as
\emph{Mahler partitions} (see \cite{Knuth:1966:almost, Bruijn:1948:mahler,
  Pennington:1953:Mahler-part-prob, Mahler:1940:spec-functional-eq}).

The number of such partitions exhibits interesting periodic fluctuations. The
situation changes, however, when \emph{compositions} into powers of $2$ are considered,
i.e., when the summands are arranged in an order.  In other words, we consider
representations of the form~\eqref{eq:partition} without further restrictions
on the exponents $a_1$, $a_2$, \ldots, $a_n$ other than being nonnegative.

% ------------------------------------------------------------------------------

\medskip

Motivated by the study of the exponential sum 
\begin{equation*}
  s(\xi) = \sum_{r=1}^{\tau} \xi^{2^r},
\end{equation*}
where $\xi$ is a primitive $q$th root of unity and $\tau$ is the order of $2$
modulo $q$ (see \cite{Molteni:2010:canc-exp-sum}),
Molteni~\cite{Molteni:2012:repr-2-powers-asy} recently studied the maximum
number of representations a positive integer can have as an ordered sum of $n$
powers of $2$. More generally, fix an integer $b\geq2$, let
\begin{equation}\label{eq:def-Ub}
  \f{\U_b}{\ell,n} = \card*{\set*{(a_1,a_2,\ldots,a_n) \in \N_0^n}{%
      b^{a_1} + b^{a_2} + \dots + b^{a_n} = \ell}}
\end{equation}
be the number of representations of $\ell$ as an ordered sum of $n$ powers of
$b$, and let $\f{\W_b}{s,n}$ be the maximum of $\f{\U_b}{\ell,n}$ over all
positive integers $\ell$ with $b$\nbd-ary sum of digits equal to $s$. It was
shown in \cite{Molteni:2010:canc-exp-sum} that
\begin{equation}\label{eq:ws_equation}
  \frac{\f{\W_2}{s,n}}{n!} 
  = \sum_{\substack{k_1,k_2,\ldots,k_s \geq 1 \\ k_1+k_2+\cdots+k_s = n}} 
  \prod_{j=1}^s \frac{\W_2(1,k_j)}{k_j!},
\end{equation}
which generalizes in a straightforward fashion to arbitrary bases~$b$. So
knowledge of $\f{\W_b}{1,n}$ is the key to understanding $\f{\W_b}{s,n}$ for
arbitrary $s$.

% ------------------------------------------------------------------------------

\medskip

For the moment, let us consider the case $b=2$. There is an equivalent
characterisation of $\f{\W_2}{1,n}$ in terms of compositions of $1$. To this
end, note that the number of representations of $2^h\ell$ as a sum of powers
of $2$ is the same as the number of representations of $\ell$ for all integers $h$
if negative exponents are allowed as well (simply multiply/divide everything by
$2^h$). Therefore, $\f{\W_2}{1,n}$ is also the number of solutions to the
Diophantine equation
\begin{equation}\label{eq:maineq}
  2^{-k_1} + 2^{-k_2} + \cdots + 2^{-k_n} = 1
\end{equation}
with nonnegative integers $k_1,k_2,\ldots,k_n$, i.e., the number of
\emph{compositions} of $1$ into powers of $2$. This sequence starts with
\begin{equation*}
  1, 1, 3, 13, 75, 525, 4347, 41245, 441675, 5259885, 68958747, \dots
\end{equation*}
and is \OEIS{A007178} in the On-Line Encyclopedia of Integer
Sequences~\cite{OEIS:2014}.

% ------------------------------------------------------------------------------

The main goal of this paper is to determine precise asymptotics for the number
of such binary compositions as $n \to \infty$. Lehr, Shallit and
Tromp~\cite{Lehr-Shallit-Tromp:1996:vec-spc-autom-reals} encountered these
compositions in their work on automatic sequences and gave a first bound,
namely
\begin{equation*}
  \W_2(1,n)/n! \leq K \cdot 1.8^n
\end{equation*}
for some constant $K$. It was mainly based on an asymptotic formula for the
number of \emph{partitions} of $1$ into powers of $2$, which was derived
independently in different contexts, cf.\@ \cite{Boyd:1975,
  Flajolet-Prodinger:1987:level, Komlos-Moser-Nemetz:1984} (or see
the recent paper of Elsholtz, Heuberger and
Prodinger~\cite{Elsholtz-Heuberger-Prodinger:2013:huffm} for a detailed
survey). This bound was further improved by Molteni, who gave the inequalities
\begin{equation*}
  0.3316 \cdot (1.1305)^n \leq \W_2(1, n)/n! \leq  (1.71186)^{n-1} \cdot n^{-1.6}
\end{equation*}
in \cite{Molteni:2010:canc-exp-sum}. Giorgilli and
Molteni~\cite{Giorgilli-Molteni:2013:repr-2-powers-rec} provided an efficient
recursive formula for $\W_2(1,n)$ and used it to prove an intriguing congruence
property. In his recent paper~\cite{Molteni:2012:repr-2-powers-asy}, Molteni
succeeded in proving the following result, thus also disproving a conjecture of
Knuth on the asymptotic behaviour of $\W_2(1,n)$.

% ------------------------------------------------------------------------------

\begin{thm}[Molteni~\cite{Molteni:2012:repr-2-powers-asy}]\label{thm:molteni}
The limit
\begin{equation*}
  \gamma = \lim_{n \to \infty} (\W_2(1,n)/n!)^{1/n} 
  = 1.192674341213466032221288982528755\ldots
\end{equation*}
exists.
\end{thm}

% ------------------------------------------------------------------------------

Molteni's argument is quite sophisticated and involves the study of the
spectral radii of certain matrices. The aim of this paper will be to present a
different approach to the asymptotics of $\W_2(1,n)$ (and more generally,
$\f{\W_2}{s,n}$) by means of generating functions that allows us to obtain more
precise information. Our main theorem reads as follows.

% ------------------------------------------------------------------------------

\begin{thm}\label{thm:asymptotics}
  There exist constants $\alpha = 0.2963720490\dots$, $\gamma =
  1.1926743412\dots$ (as in Theorem~\ref{thm:molteni}) and $\kappa =
  2/(3\gamma) < 1$ such that
  \begin{equation*}
    \frac{\W_2(1,n)}{n!} = \alpha \gamma^{n-1} (1 + \Oh{\kappa^n}).
  \end{equation*}
  More generally, for every fixed $s$, there exists a polynomial $\f{P_s}{n}$
  with leading term
  \begin{equation*}
    (\alpha/\gamma)^s n^{s-1}/(s-1)!
  \end{equation*}
  such that
  \begin{equation*}
    \frac{\f{\W_2}{s,n}}{n!} = \f{P_s}{n} \gamma^{n} (1 + \Oh{\kappa^n}).
  \end{equation*}
\end{thm}

% ------------------------------------------------------------------------------

We also prove a more general result for arbitrary bases instead of $2$.
Consider the Diophantine equation
\begin{equation}\label{eq:maineq_general}
  b^{-k_1} + b^{-k_2} + \cdots + b^{-k_n} = 1.
\end{equation}
Multiplying by the common denominator and taking the equation modulo $b-1$, we
see that there can only be solutions if $n \equiv 1 \bmod b-1$, i.e., $n =
(b-1)m+1$ for some nonnegative integer $m$. We write $q_b(m)$ for the number of
solutions ($n$-tuples of nonnegative integers
satisfying~\eqref{eq:maineq_general}) in this case. Note that $q_b(m)$ is also
the maximum number of representations of an arbitrary power of $b$ as an
ordered sum of $n = (b-1)m+1$ powers of $b$. We have the following general
asymptotic formula.

% ------------------------------------------------------------------------------

\begin{thm}\label{thm:asymptotics:general_base}
  For every positive integer $b \geq 2$, there exist constants $\alpha =
  \alpha_b$, $\gamma = \gamma_b$ and $\kappa = \kappa_b < 1$ such that the
  number $q_b(m)$ of compositions of $1$ into $n = (b-1)m+1$ powers of $b$,
  which is also the maximum number $\f{\W_b}{1,n}$ of representations of a
  power of $b$ as an ordered sum of $n$ powers of $b$, satisfies
  \begin{equation*}
    \frac{\f{\W_b}{1,n}}{n!} = \frac{q_b(m)}{n!} 
    = \alpha \gamma^m (1 + \Oh{\kappa^m}).
  \end{equation*}
  More generally, the maximum number $\f{\W_b}{s,n}$ of representations of a
  positive integer with $b$-ary sum of digits $s$ as an ordered sum of $n =
  (b-1)m+s$ powers of $b$ is asymptotically given by
  \begin{equation*}
    \frac{\f{\W_b}{s,n}}{n!} = \f{P_{b,s}}{m} \gamma^m (1 + \Oh{\kappa^m}),
  \end{equation*}
  where $\f{P_{b,s}}{m}$ is a polynomial with leading term $\alpha^s
  m^{s-1}/(s-1)!$.
\end{thm}

% ------------------------------------------------------------------------------

The key idea is to equip every \emph{partition} of $1$ into powers of $2$ (or
generally $b$) with a weight that essentially gives the number of ways it can
be permuted to a composition, and to apply the recursive approach that was used
to count partitions of $1$: if $p_2(n)$ denotes the number of such partitions
into $n$ summands, then the remarkable generating function identity
\begin{equation}\label{eq:partitionsgf}
  \sum_{n=1}^{\infty} p_2(n)x^n = 
  \frac{\sum_{j=0}^{\infty} (-1)^j x^{2^j-1} 
    \prod_{i=1}^j \frac{x^{2^i-1}}{1-x^{2^i-1}}}{\sum_{j=0}^{\infty} (-1)^j 
    \prod_{i=1}^j \frac{x^{2^i-1}}{1-x^{2^i-1}}}
\end{equation}
holds, and this can be generalised to arbitrary bases $b$, see the recent paper of
Elsholtz, Heuberger and
Prodinger~\cite{Elsholtz-Heuberger-Prodinger:2013:huffm}. In our case, we do
not succeed to obtain a similarly explicit formula for the generating function,
but we can write it as the quotient of two determinants of infinite matrices
and infer analytic information from it. The paper is organised as follows: we
first describe the combinatorial argument that yields the generating function,
a priori only within the ring of formal power series. We then study the
expression obtained for the generating function in more detail to show that it
can actually be written as the quotient of two entire functions. The rest of
the proof is a straightforward application of residue calculus (using the
classical Flajolet--Odlyzko singularity
analysis~\cite{Flajolet-Odlyzko:1990:singul}).

% ------------------------------------------------------------------------------

Furthermore, we consider the maximum of $\f{\U_b}{\ell,n}$ over all $\ell$, for
which we write
\begin{equation*}
  \f{M_b}{n} = \max_{\ell \geq 1}\, \f{\U_b}{\ell,n} 
  = \max_{s \geq 1}\, \f{\W_b}{s,n}.
\end{equation*}
This means that $\f{M_b}{n}$ is the maximum possible number of representations
of a positive integer as a sum of exactly $n$ powers of $b$. Equivalently,
it is the largest coefficient in the power series expansion of
\begin{equation*}
  \big( x + x^b + x^{b^2} + \cdots \big)^n.
\end{equation*}
When $b=2$, Molteni~\cite{Molteni:2012:repr-2-powers-asy} obtained the
following bounds for this quantity:
\begin{equation*}
  (1.75218)^n \ll \f{M_2}{n}/n! \leq (1.75772)^n.
\end{equation*}
The gap between the two estimates is already very small; we improve this a
little further by providing the constant of exponential growth as well as a
precise asymptotic formula.

% ------------------------------------------------------------------------------

\begin{thm}\label{thm:maximum}
  For a certain constant $\nu = 1.7521819\ldots$ (defined precisely in
  Section~\ref{sec:thm-max}), we have
  \begin{equation*}
    \f{M_2}{n}/n! \leq \nu^n
  \end{equation*}
  for all $n \geq 1$, and the constant is optimal: we have the more precise
  asymptotic formula
  \begin{equation*}
    \f{M_2}{n}/n! \sim \lambda n^{-1/2} \nu^n
  \end{equation*}
  with $\lambda = 0.2769343\ldots$.
\end{thm}

% ------------------------------------------------------------------------------

Again, Theorem~\ref{thm:maximum} holds for arbitrary integer bases $b\geq2$ for
some constants $\nu=\nu_b$ and $\lambda=\lambda_b$ (it will be explained
precisely how they are obtained). This is formulated as
Theorem~\ref{thm:maximum-full} in Section~\ref{sec:thm-max}.

% ------------------------------------------------------------------------------

The final section contains the analysis of some parameters. We study the
exponent of the largest denominator and the number of distinct parts in a
composition of~$1$. In both cases a central limit theorem is shown; mean and
variance are linear in the number of summands, cf.\@ Theorems~\ref{thm:largest}
and~\ref{thm:distinct}.

% ------------------------------------------------------------------------------

\section{The Recursive Approach}
\label{sec:rec}

For our purposes, it will be most convenient to work in the setting of
compositions of $1$, i.e., we are interested in the number $q_b(m)$ of
(ordered) solutions to the Diophantine equation~\eqref{eq:maineq_general},
where $n = (b-1)m+1$, as explained in the introduction. Our first goal is to
derive a recursion for $q_b(m)$ and some related quantities, which leads to a
system of functional equations for the associated generating functions.

Let $\vc{k} = (k_1,k_2,\ldots,k_n)$ be a solution to the Diophantine
equation~\eqref{eq:maineq_general} with $k_1 \geq k_2 \geq \cdots \geq k_n$. We
will refer to such an $n$-tuple as a ``partition'' (although technically the
$k_i$ are only the exponents in a partition). We denote by $\co(\vc{k})$ the
number of ways to turn it into a composition. If $w_0$ is the number of zeros,
$w_1$ the number of ones, etc.\@ in $\vc{k}$, then we clearly have
\begin{equation*}
\co(\vc{k}) = \frac{n!}{\prod_{j \geq 0} w_j!}.
\end{equation*}
The \emph{weight} of a partition $\vc{k}$, denoted by $\wt(\vc{k})$, is now
simply defined as
\begin{equation*}
\wt(\vc{k}) = \frac{1}{\prod_{j \geq 0} w_j!} = \frac{\co(\vc{k})}{n!}.
\end{equation*}
Now let 
\begin{multline*}
  \Pa_m = \Big\{ \vc{k} = (k_1,k_2,\dots,k_n) \,\Big\vert\,
    \text{$n = (b-1)m+1$, } \\
    \text{$b^{-k_1}+b^{-k_2} + \dots + b^{-k_n} = 1$,
    $k_1 \geq k_2 \geq \dots \geq k_n$} \Big\}
\end{multline*}
be the set of all partitions of $1$ with $n = (b-1)m+1$ terms and, likewise,
\begin{equation*}
  \Co_m = \Big\{ \vc{k} = (k_1,k_2,\dots,k_n) \,\Big\vert\,
    \text{$n = (b-1)m+1$, $b^{-k_1}+b^{-k_2} + \dots + b^{-k_n} = 1$} \Big\}
\end{equation*}
the set of compositions. We obtain the formula
\begin{equation*}
  q_b(m) = \card*{\Co_{m}} 
  = \sum_{\vc{k} \in \Pa_m} \co(\vc{k}) = n!\, \sum_{\vc{k} \in \Pa_m} \wt(\vc{k})
\end{equation*}
for their number.

Our next step involves an important observation that is also used to obtain the
generating function~\eqref{eq:partitionsgf}. Consider an element $\vc{k}$ of
$\Pa_m$, and let $r$ be the number of times the greatest element $k_1$ occurs
(i.e., $k_1 = k_2 = \cdots = k_r > k_{r+1}$). This number must be divisible by
$b$ (as can be seen by multiplying~\eqref{eq:maineq_general} by $b^{k_1}$)
unless $\vc{k}$ is the trivial partition, so we can replace them by $r/b$
fractions with denominator $b^{k_1-1}$.

This process can be reversed. Given a partition $\vc{k}$ in which the largest
element occurs $r$ times, we can replace $s$, $1 \leq s \leq r$, of these
fractions by $bs$ fractions with denominator $b^{k_1+1}$. This recursive
construction can be illustrated nicely by a tree structure as in
Figure~\ref{fig:tree} for the case $b=2$. Each partition corresponds to a
so-called canonical tree (see \cite{Elsholtz-Heuberger-Prodinger:2013:huffm}),
and vice versa. Note that if $\vc{k} \in \Pa_m$, then the resulting partition
$\vc{k'}$ lies in $\Pa_{m+s}$, and we clearly have
\begin{equation}\label{eq:weightrel}
\wt(\vc{k'}) = \wt(\vc{k}) \cdot \frac{r!}{(r-s)!\,(bs)!}.
\end{equation}

\begin{figure}[htbp]\centering
\begin{tikzpicture}[vertex/.style={circle,draw=black,fill=black,
         inner sep=0pt,minimum size=1mm},
       edge/.style={draw=black}, xscale=0.1, yscale=0.7]
   \draw (0,0) node  (R) [vertex] {};
   \draw (-8,-1) node (R0) [vertex] {};
   \draw (8,-1) node (R1) [vertex] {};
   \draw[edge] (R)--(R0);
   \draw[edge] (R)--(R1);
   \draw (-12,-2) node (R00) [vertex, label=below:$\frac14$] {};
   \draw (-4,-2) node (R01) [vertex, label=below:$\frac14$] {};
   \draw (4,-2) node (R10) [vertex, label=below:$\frac14$] {};
   \draw (12,-2) node (R11) [vertex] {};
   \draw[edge] (R0)--(R00);
   \draw[edge] (R0)--(R01);
   \draw[edge] (R1)--(R10);
   \draw[edge] (R1)--(R11);
   \draw (10,-3) node (R110) [vertex, label=below:$\frac18$] {};
   \draw (14,-3) node (R111) [vertex, label=below:$\frac18$] {};
   \draw[edge] (R11)--(R110);
   \draw[edge] (R11)--(R111);
 \end{tikzpicture}
 \caption{The canonical tree associated with the partition $1 =
   3\cdot2^{-2} + 2\cdot 2^{-3}$ of $1$ into powers of $2$. This partition has weight $\frac1{12}$ and corresponds to $10$ distinct compositions.}\label{fig:tree}
\end{figure}
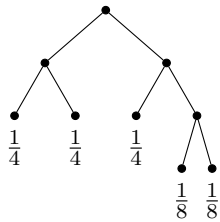

Now we can turn to generating functions. Let $\Pa_{m,r}$ be the subset of
$\Pa_m$ that only contains partitions for which $k_1 = k_2 = \cdots = k_r >
k_{r+1}$ (i.e., in~\eqref{eq:maineq_general}, the largest exponent occurs
exactly $r$ times), and let $\Co_{m,r}$ be the set of compositions obtained by
permuting the terms of an element of $\Pa_{m,r}$. We define a generating
function by
\begin{equation*}
  \f{Q_r}{x} = \sum_{m \geq 0} \frac{\card*{\Co_{m,r}}}{((b-1)m+1)!} x^m 
  = \sum_{m \geq 0} \sum_{\vc{k} \in \Pa_{m,r}} \frac{\co(\vc{k})}{((b-1)m+1)!} x^m 
  = \sum_{m \geq 0} \sum_{\vc{k} \in \Pa_{m,r}} \wt(\vc{k}) x^m.
\end{equation*}
We have $\f{Q_1}{x} = 1$ and $\f{Q_r}{x} = 0$ for all other $r$ not divisible by $b$. Moreover, for all $s \geq 1$ the recursive relation described above and in
particular~\eqref{eq:weightrel} yield
\begin{equation}\label{eq:genfun_id}
  \begin{split}
    \f{Q_{bs}}{x} &= \sum_{m \geq 0} \sum_{\vc{k'} \in \Pa_{m,bs}} \wt(\vc{k'})
    x^m = \sum_{r \geq s} \sum_{m \geq s} \sum_{\vc{k} \in \Pa_{m-s,r}} \wt(\vc{k})
    \frac{r!}{(r-s)!\,(bs)!}x^m \\
    &= x^s \sum_{r \geq s} \frac{r!}{(r-s)!\,(bs)!}  \sum_{m \geq s}
    \sum_{\vc{k} \in \Pa_{m-s,r}} \wt(\vc{k})x^{m-s} = x^s \sum_{r \geq s}
    \frac{r!}{(r-s)!\,(bs)!} \f{Q_r}{x}.
  \end{split}
\end{equation}
This can be seen as an infinite system of linear equations. Define the infinite
(column-)vector $\vc{V}(x) =
(\f{Q_b}{x},\f{Q_{2b}}{x},\f{Q_{3b}}{x},\ldots)^T$, and the infinite matrix
$\vc{M}(x)$ by its entries
\begin{equation*}
  m_{ij} = \begin{cases} 
    \frac{(bj)!\,x^i}{(bj-i)!\,(bi)!} & \text{if } i \leq bj, \\ 
    0 & \text{otherwise.}
  \end{cases}
\end{equation*}
Then the identity~\eqref{eq:genfun_id} above turns into the matrix identity
\begin{equation}\label{eq:matrixid}
  \vc{V}(x) = \vc{M}(x) \vc{V}(x) + \frac{x}{b!} \vc{e}_1,
\end{equation}
where $\vc{e}_1 = (1,0,0,\ldots)^T$ denotes the first unit vector. Within the
ring of formal power series, this readily yields
\begin{equation}\label{eq:vexplicit}
  \vc{V}(x) = \frac{x}{b!} (\vc{I}-\vc{M}(x))^{-1} \vc{e}_1,
\end{equation}
and the generating function 
\begin{equation*}
  \f{Q}{x} = \sum_{r \geq 1} \f{Q_r}{x} 
  = \sum_{m \geq 0} \frac{q_b(m)}{((b-1)m+1)!} x^m
\end{equation*}
(recall that $q_b(m)$ is the number of compositions of $1$ into $n = (b-1)m+1$
powers of $b$) is given by
\begin{equation*}
  \f{Q}{x} = 1 + \vc{1}^T \vc{V}(x) 
  = 1 + \frac{x}{b!} \vc{1}^T (\vc{I}-\vc{M}(x))^{-1} \vc{e}_1.
\end{equation*}
For our asymptotic result, we will need the dominant singularity of $\f{Q}{x}$,
i.e., the zero of $\det(\vc{I}-\vc{M}(x))$ that is closest to $0$. A priori, it is not
even completely obvious that this determinant is well-defined, but the
reasoning is similar to a number of comparable problems.

As mentioned earlier, the determinant $\f{T}{x} = \det(\vc{I}-\vc{M}(x))$
exists a priori within the ring of formal power series, as the limit of the
principal minor determinants. We can write it as
\begin{equation}\label{eq:detexpression}
  \det(\vc{I}-\vc{M}(x)) = \sum_{h\geq0} (-1)^h 
  \sum_{\substack{1\leq i_1<i_2<\dots<i_h \\ i_1,\dots,i_h\in\N}} x^{i_1+i_2+\cdots+i_h} 
  \sum_{\sigma} (\sgn \sigma) 
  \prod_{k=1}^h \frac{(b\f{\sigma}{i_k})!}{(b\f{\sigma}{i_k}-i_k)!\,(bi_k)!},
\end{equation}
where the inner sum is over all permutations~$\sigma$ of
$\{i_1,i_2,\ldots,i_h\}$.  Using Eaves' sufficient condition,
cf.~\cite{Eaves:1970}, we get at least convergence for $\abs{x}<1$.

We can even
show that the formal power series~$T$ given by~\eqref{eq:detexpression} defines an
entire function. This is proven in Section~\ref{sec:bounds}. 
The same is true (by the same argument) for
\begin{equation*}
  \f{S}{x} = \vc{1}^T \operatorname{adj}(\vc{I}-\vc{M}(x)) \vc{e}_1 
  = \det(\vc{M}^*(x)),
\end{equation*}
where $\vc{M}^*$ is obtained from $\vc{I}-\vc{M}(x)$ by replacing the first row
by $\vc{1}$. Hence we can write the generating function $\f{Q}{x}$ as
\begin{equation}\label{eq:gen-A-num-den}
  \f{Q}{x} = 1 + \frac{x}{b!} \frac{\f{S}{x}}{\f{T}{x}},
\end{equation}
where $\f{S}{x}$ and $\f{T}{x}$ are both entire functions. The singularities of
$\f{Q}{x}$ are thus all poles, and it remains to determine the dominant
singularity, i.e., the zero of $T(x) = \det(\vc{I}-\vc{M}(x)) $ with smallest
modulus.

% ------------------------------------------------------------------------------

\section{Bounds and Entireness}
\label{sec:bounds}

% ------------------------------------------------------------------------------

In this section the two formal power series
\begin{equation*}
  \f{T}{x} = \sum_{n\geq 0} t_n x^n
  = \det(I-\vc{M}(x))
\end{equation*}
and
\begin{equation*}
  \f{S}{x} = \sum_{n\geq 0} s_n x^n
  = \vc{1}^T \operatorname{adj}(I-\vc{M}(x)) \vc{e}_1
\end{equation*}
of Section~\ref{sec:rec} (in particular cf.\ Equations~\eqref{eq:detexpression}
and~\eqref{eq:gen-A-num-den}) are analyzed. Other (similar) functions arising
on the way can be dealt with in a similar fashion.

Note that $\f{S}{x}$ is the
determinant of a matrix, which is obtained by replacing the first row of
$I-\vc{M}(x)$ by $\vc{1}$.

% ------------------------------------------------------------------------------

We find bounds for the coefficients $t_n$ and $s_n$, which will be needed for
numerical calculations with guaranteed error estimates as well. Further, those
bounds will tell us that the two functions $\f{T}{x}$ and $\f{S}{x}$ are
entire.

% ------------------------------------------------------------------------------

\begin{lemma}\label{lem:bound-det-coeff}
  The coefficients $t_n$ satisfy the bound
  \begin{equation*}
    \abs{t_n} \leq \exp\left(-\frac{b-1}{2}n\log n - cn + n\f{g}{n} \right)
  \end{equation*}
  with $c = (b-1)\left(\log\frac{b-1}{\sqrt{2}}-1\right)$ and with a
decreasing function $\f{g}{n}$, which tends to zero as
  $n\to\infty$. In particular, the formal power series $T$ defines an entire
  function. The same is true for the formal power series~$S$. More precisely,
  we have
  \begin{equation*}
    \abs{s_n} \leq \left((b-1)!+1\right) 
    \exp\left(-\frac{b-1}{2}n\log n - cn + (n+1) \f{g}{n} \right).
  \end{equation*}
\end{lemma}

% ------------------------------------------------------------------------------

\begin{proof}
  Recall expression~\eqref{eq:detexpression} for the determinant, namely
  \begin{equation*}
    \det(\vc{I}-\vc{M}(x)) = \sum_{h\geq0} (-1)^h 
    \sum_{\substack{1\leq i_1<i_2<\dots<i_h \\ i_1,\dots,i_h\in\N}} x^{i_1+i_2+\cdots+i_h} 
    \sum_{\sigma} (\sgn \sigma) 
    \prod_{k=1}^h \frac{(b\f{\sigma}{i_k})!}{(b\f{\sigma}{i_k}-i_k)!\,(bi_k)!}.
  \end{equation*}
  Write
  $n = i_1 + i_2 + \cdots + i_h$ for the exponent of $x$, and note that
  \begin{equation*}
    \prod_{k=1}^h \frac{(b\f{\sigma}{i_k})!}{(bi_k)!} = 1,
  \end{equation*}
  which is independent of the permutation~$\sigma$. We also have
  \begin{equation*}
    \sum_{k = 1}^h (b\f{\sigma}{i_k}-i_k) = (b-1)\sum_{k = 1}^h i_k = (b-1)n.
  \end{equation*}
  Since $a! \geq \exp(a(\log a - 1))$ for all positive integers $a$ and $f(x) =
  x(\log x - 1)$ is a convex function, we have
  \begin{align*}
    \prod_{k = 1}^h (b\f{\sigma}{i_k}-i_k)!
    &\geq \exp \left( \sum_{k=1}^h (b\f{\sigma}{i_k}-i_k) 
      \left( \log (b\f{\sigma}{i_k}-i_k) - 1 \right) \right) \\
    &\geq \exp \left( h \frac{(b-1)n}{h} 
      \left( \log \frac{(b-1)n}{h} - 1 \right) \right) \\
    &= \exp \left( (b-1)n \left( \log \frac{(b-1)n}{h} - 1 \right) \right).
  \end{align*}
  Since $i_1,i_2,\ldots,i_h$ have to be distinct, we also have
  \begin{equation*}
    n = i_1 + i_2 + \cdots + i_h \geq 1 + 2 + \cdots + h = \frac{h(h+1)}{2} \geq \frac{h^2}{2}.
  \end{equation*}
  Thus $h \leq \sqrt{2n}$, which means that
  \begin{equation*}
    \prod_{k = 1}^h (b\f{\sigma}{i_k}-i_k)!
    \geq \exp \left( \frac{b-1}{2}n\log n + 
      (b-1)n \left(\log\frac{b-1}{\sqrt{2}} - 1\right) \right).
  \end{equation*} 
  Now that we have an estimate for each term in~\eqref{eq:detexpression}, let
  us also determine a bound for the number of terms corresponding to each
  exponent $n$. 

  It is well known that the number of partitions $q(n)$ of $n$ into distinct
  parts is asymptotically equal to $\exp\!\big( \pi \sqrt{n/3} + \Oh{\log n}
  \big)$. In Robbins's
  paper~\cite{Robbins:2007:upper-bound-partitions-distinct} we can find the
  explicit upper bound\footnote{Note that in the published version
    of~\cite{Robbins:2007:upper-bound-partitions-distinct} a constant in the
    main theorem is printed incorrectly.}
  \begin{equation*}
%    \frac{\pi}{\sqrt{12(n-1)}} 
%    \exp\left(\frac{\pi}{\sqrt{3}}\sqrt{n-1} + \frac{\pi^2}{12}\right)
    q(n) \leq \frac{\pi}{\sqrt{12n}} 
    \exp\left(\frac{\pi}{\sqrt{3}}\sqrt{n} + \frac{\pi^2}{12}\right).
  \end{equation*}

  For each choice of $\{i_1,i_2,\ldots,i_h\}$, there are at most $h!$
  permutations~$\sigma$ that contribute, which can be bounded by means of
  Stirling's formula (using also $h \leq \sqrt{2n}$ again). This gives
  \begin{equation*}
    h! \leq \exp\left( h \log h - h  + \tfrac12 \log h + 1 \right) 
    \leq \exp \left((\sqrt{2n} + \tfrac12)\log(\sqrt{2n}) - \sqrt{2n}+1\right).
  \end{equation*}
  It follows that the coefficient $t_n$ of $T$ is bounded (in absolute values)
  by
  \begin{multline*}
    \frac{\exp\left(    
        \frac{\pi}{\sqrt{3}}\sqrt{n} + \frac{\pi^2}{12}
        + \log\pi - \frac12 \log (12n) 
        + (\sqrt{2n} + \tfrac12)\log(\sqrt{2n}) - \sqrt{2n}+1
      \right)}{\exp\left(   
        \frac{b-1}{2}n\log n + 
        (b-1)n \left(\log\frac{b-1}{\sqrt{2}} - 1\right)
      \right)} \\
    =\exp\left(-\frac{b-1}{2}n\log n - cn + \Oh{\sqrt{n} \log n} \right),
  \end{multline*}
  which proves the theorem for a suitable choice of $g(n)$. A possible explicit
  bound (relevant for our numerical calculations, see
  Section~\ref{sec:numerical}) is
  \begin{equation*}
    \abs{t_n} \leq \exp\left(-\frac{b-1}{2}n\log n - cn 
      + \sqrt{\frac{n}2}\log n + \sqrt{n} + 3\right).
  \end{equation*}
  Since this bound decays superexponentially, the determinant $T
  =\det(\vc{I}-\vc{M}(x))$ is an entire function.

  The same argument works for $S$. There, we split up into the summands
  where we have $i_1=1$ and all other summands. For the second part (the summands
  with $i_1>1$), the terms are the same as in the determinant that defines $T$, so it is bounded by the same expression.
Each of the summands with $i_1=1$ equals a summand of $\det(I-\vc{M}(x))$ multiplied by the factor
  \begin{equation*}
    -\frac{(b\f{\sigma}{i_1}-i_1)!\,(bi_1)!}{(b\f{\sigma}{i_1})!\, x}
    = -\frac{b!}{x} \frac{(b\f{\sigma}{1}-1)!}{(b\f{\sigma}{1})!}
    = -\frac{(b-1)!}{x \f{\sigma}{1}}
  \end{equation*}
  or is zero (when $\f{\sigma}{i_1}=1$). Therefore, the sum of these  terms can be bounded by $(b-1)!$ times the bound we 
obtained for the coefficient of $x^{n+1}$ in $\det(I-\vc{M}(x))$. This gives us
\begin{align*}
|s_n| &\leq \exp\left(-\frac{b-1}{2}n\log n - cn + n \f{g}{n} \right) \\
&\qquad + (b-1)!\exp\left(-\frac{b-1}{2}(n+1)\log (n+1) - c(n+1) + (n+1) \f{g}{n+1} \right) \\
&\leq (1+(b-1)!)\exp\left(-\frac{b-1}{2}n\log n - cn + (n+1) \f{g}{n} \right),
\end{align*}
which completes the proof.
\end{proof}

% ------------------------------------------------------------------------------

Lemma~\ref{lem:bound-det-coeff} immediately yields a simple estimate for the tails of the power series $S$ and $T$.

\begin{lemma}\label{lem:bound-det-tail}
  Let $N\in\N$ and $x\in\C$, and let $c$ and $\f{g}{n}$ be as in
  Lemma~\ref{lem:bound-det-coeff}. Set
  \begin{equation*}
    q = \frac{e^{\f{g}{N}} \abs{x}}{e^c \sqrt{N^{b-1}}}
  \end{equation*}
  and suppose that $q<1$.  Then we have the inequality
  \begin{equation*}
    \abs[Big]{\sum_{n \geq N} t_n x^n} \leq \frac{q^N}{1-q}
  \end{equation*}
  for the tails of the infinite sum in the determinant~$T$. For the tails of
  the determinant $S$, we have the analogous inequality
  \begin{equation*}
    \abs[Big]{\sum_{n \geq N} s_n x^n} \leq ((b-1)!+1)e^{g(N)}\frac{q^N}{1-q}.
  \end{equation*}
\end{lemma}

% ------------------------------------------------------------------------------

\begin{proof}
  By Lemma~\ref{lem:bound-det-coeff} we have
  \begin{equation*}
    \abs{t_n} \leq \exp\Big(-\frac{b-1}{2}n\log n - cn + n \f{g}{n} \Big).
  \end{equation*}
Now we use monotonicity to obtain
  \begin{equation*}
    \abs[Big]{\sum_{n \geq N} t_n x^n}
    \leq \sum_{n \geq N} 
    \left(\frac{e^{\f{g}{n}} \abs{x}}{e^c \sqrt{n^{b-1}}}\right)^n
    \leq \sum_{n \geq N} 
    \left(\frac{e^{\f{g}{N}} \abs{x}}{e^c \sqrt{N^{b-1}}}\right)^n 
    = q^N \frac{1}{1-q}.
  \end{equation*}
The second inequality follows in the same way.
\end{proof}

% ------------------------------------------------------------------------------

\section{Analyzing the Generating Function}
\label{sec:poles-gf}

% ------------------------------------------------------------------------------

Infinite systems of functional equations appear quite frequently in the
analysis of combinatorial problems, see for example the recent work of Drmota,
Gittenberger and
Morgenbesser~\cite{Drmota-Gittenberger-Morgenbesser:2012:inf-systems}. Alas,
their very general theorems are not applicable to our situation as the infinite
matrix~$\vc{M}$ does not represent an $\ell_p$-operator (one of their main
requirements), due to the fact that its entries increase (and tend to $\infty$)
along rows. However, we can adapt some of their ideas to our setting.

% ------------------------------------------------------------------------------

The main result of this section is the following lemma.

% ------------------------------------------------------------------------------

\begin{lemma}\label{lem:poles}
  For every $b \geq 2$, the generating function $\f{Q}{x}$ has a simple pole at
  a positive real point $\rho_b$ and no other poles with modulus $< \rho_b +
  \epsilon_b$ for some $\epsilon_b > 0$.
\end{lemma}

% ------------------------------------------------------------------------------

\begin{proof}[Proof of Lemma~\ref{lem:poles}]
  First of all, we rule out the possibility that $\f{Q}{x}$ is
  entire by providing a lower bound for the coefficients $q_b(m)$. To this end,
  consider compositions of $1$ consisting of $b-1$ copies of
  $b^{-1},b^{-2},\ldots,b^{1-m}$ and $b$ copies of $b^{-m}$. Since there are
  $\frac{((b-1)m+1)!}{((b-1)!)^{m-1}b!}$ possible ways to arrange them in an
  order, we know that
  \begin{equation*}
    q_b(m) \geq \frac{((b-1)m+1)!}{((b-1)!)^{m-1}b!},
  \end{equation*}
  from which it follows that the radius of convergence of $\f{Q}{x}$ is at most
  $(b-1)!$. Since all coefficients are positive, Pringsheim's theorem
  guarantees that the radius of convergence, which we denote by $\rho_b$, is
  also a singularity.

  We already know that $\f{Q}{x}$ is meromorphic (being the
  quotient of two entire functions), hence $\rho_b$ is a pole singularity. Let
  $p$ be the pole order, and with $Q_{br}$ as in Section~\ref{sec:rec} set
  \begin{equation*}
    w_r = \lim_{x \to \rho_b^-} (\rho_b-x)^p
    \f{Q_{br}}{x},
  \end{equation*}
  which must be nonnegative and real. Moreover, we have
  \begin{equation*}
    \f{Q_b}{x} = \frac{x}{b!} + x \sum_{r \geq 1} 
    \frac{r}{(b-1)!} \f{Q_{br}}{x} 
    \geq \frac{x}{b!} \Big(1  + \sum_{r \geq 1} \f{Q_{br}}{x} \Big)
    = \frac{x}{b!} \f{Q}{x},
  \end{equation*}
  which shows that $w_1$ is even strictly positive. Multiplying the matrix
  equation~\eqref{eq:matrixid} by $(\rho_b-x)^p$ and taking the limit, we see
  that $\vc{w} = (w_1,w_2,\ldots)^T$ is a right eigenvector of
  $\vc{M}(\rho_b)$. Since all entries in $\vc{M}(\rho_b)$ are nonnegative and those on
  and above the main diagonal are strictly positive, it follows that $w_r > 0$
  for all $r$, i.e., all functions $\f{Q_r}{x}$ have the same pole order (as
  $\f{Q}{x}$).

  Now we split the identity~\eqref{eq:matrixid}. Let $m_{11} = x/(b-1)!$ be the
  first entry of $\vc{M}(x)$, $\vc{c}$ the rest of the first column, $\vc{r}$
  the rest of the first row and $\overline{\vc{M}}$ the matrix obtained from
  $\vc{M}$ by removing the first row and the first column. Moreover,
  $\overline{\vc{V}}$ is obtained from $\vc{V}$ by removing the first entry
  $\f{Q_b}{x}$. Now we have
  \begin{equation}\label{eq:a1eq}
    \f{Q_b}{x} = m_{11} \f{Q_b}{x} 
    + \vc{r}\,  \overline{\vc{V}} + \frac{x}{b!}
  \end{equation}
  and
  \begin{equation*}
    \overline{\vc{V}} = \vc{c}  \f{Q_b}{x} 
    + \overline{\vc{M}}\, \overline{\vc{V}},
  \end{equation*}
  from which we obtain
  \begin{equation}\label{eq:vexpl}
    \overline{\vc{V}} = (\vc{I} - \overline{\vc{M}})^{-1} 
     \vc{c}  \f{Q_b}{x}.
  \end{equation}
  Once again, the inverse $(\vc{I} - \overline{\vc{M}})^{-1}$ exists a priori
  in the ring of formal power series, but one can show that $\det(\vc{I} -
  \overline{\vc{M}})$ is in fact an entire function, so the entries of the
  inverse are all meromorphic (see again the calculations in
  Section~\ref{sec:bounds}).
  Moreover, $(\vc{I} - \overline{\vc{M}})^{-1} \vc{c}$ cannot have a
  singularity at $\rho_b$ or at any smaller positive real number, because if
  this was the case, the right hand side of~\eqref{eq:vexpl} would have a higher
  pole order at that point than the left hand side. Since it has positive
  coefficients only (the inverse can be expanded into a geometric series), its
  entries must be analytic in a circle of radius $> \rho_b$ around $0$. Now we
  substitute~\eqref{eq:vexpl} in~\eqref{eq:a1eq} to obtain
  \begin{equation*}
    \f{Q_b}{x}  = m_{11} \f{Q_b}{x} + \vc{r}  (\vc{I} -
    \overline{\vc{M}})^{-1}  \vc{c}  \f{Q_b}{x} + \frac{x}{b!}
  \end{equation*}
  and thus
  \begin{equation*}
    \f{Q_b}{x} = \frac{x}{b!} \left( 1 - m_{11} - \vc{r} 
       (\vc{I} - \overline{\vc{M}})^{-1}  \vc{c} \right)^{-1}.
  \end{equation*}
  Note that
  \begin{equation*}
    \f{R}{x} = m_{11} + \vc{r}  (\vc{I} 
    - \overline{\vc{M}})^{-1}  \vc{c}
  \end{equation*}
  has only positive coefficients, so $\f{R}{x} = 1$ has a unique positive real
  solution, which must be $\rho_b$ (recalling that $\f{R}{x}$ is analytic in a
  circle of radius $> \rho_b$ around $0$). Of course, $\f{R'}{\rho_b} > 0$, so
  its multiplicity is $1$, which means that $\rho_b$ is a simple
  pole. Moreover, by the triangle inequality there are no complex solutions of
  $\f{R}{x} = 1$ with the same modulus, which means that there are no further
  singularities of $\f{Q_b}{x}$ (and thus $\f{Q}{x}$) in a circle of radius
  $\rho_b+\epsilon_b$ around $0$ for suitable $\epsilon_b > 0$.
\end{proof}

% ------------------------------------------------------------------------------

\section{Getting the Asymptotics}
\label{sec:asymptotics}

% ------------------------------------------------------------------------------

In this section, we prove Theorems~\ref{thm:asymptotics}
and~\ref{thm:asymptotics:general_base}, which give us constants $\alpha_b$,
$\gamma_b$ and $\kappa_b < 1$ such that for $n=(b-1)m+1$
\begin{equation*}
  \frac{\f{\W_b}{s,n}}{n!} = \f{P_{b,s}}{m} \gamma_b^m (1 + \Oh{\kappa_b^m})
\end{equation*}
holds, where $\f{P_{b,s}}{m}$ is a polynomial with leading term $\alpha_b^s
m^{s-1}/(s-1)!$. Numerical values of the $\alpha_b$ and $\gamma_b$ can be found
in Table~\ref{tab:values-thm-asy}. It is explained in the next section how
these numerical values are determined in a reliable way.

% ------------------------------------------------------------------------------

\begin{table}[htbp]
  \centering
  \begin{equation*}
    \begin{array}{ccc}
      b & \alpha & \gamma \\
      \hline
      2 & 0.296372 & 1.19268 \\
3 & 0.279852 & 0.534502 \\
4 & 0.236824 & 0.170268 \\
5 & 0.196844 & 0.0419317 \\
6 & 0.165917 & 0.00834837 \\
7 & 0.142679 & 0.00138959 \\
8 & 0.1249575 & 0.000198440 \\

      \hline
    \end{array}
  \end{equation*}
  \caption{Truncated decimal values for the constants of 
    Theorem~\ref{thm:asymptotics:general_base}. See Section~\ref{sec:numerical} for the method of computation.}
  \label{tab:values-thm-asy}
\end{table}

% ------------------------------------------------------------------------------

For easier reading, we skip the index~$b$ again, i.e., we set
$\alpha=\alpha_b$, $\gamma=\gamma_b$, and so on. The proof is the same for
all~$b$, except for the fact that different constants occur. 

% ------------------------------------------------------------------------------

\begin{proof}[Proof of Theorem~\ref{thm:asymptotics:general_base}] 
  By now,
  we know that the function $\f{Q}{x}$ can be written as the quotient of two
  entire functions, cf.\ Section~\ref{sec:rec} and
  Lemma~\ref{lem:bound-det-coeff}. More specifically, we use
  \begin{equation*}
    Q(x) = 1 + \frac{x}{b!} \frac{\f{S}{x}}{\f{T}{x}}.
  \end{equation*}
  As Lemma~\ref{lem:poles} shows, $Q(x)$ has exactly one pole~$\rho$ (which is
  a simple pole) inside some disc with radius $\rho+\epsilon$, $\epsilon>0$,
  around $0$. Thus we can directly apply singularity
  analysis~\cite{Flajolet-Odlyzko:1990:singul} in the meromorphic setting
  (cf.\@ Theorem IV.10 of~\cite{Flajolet-Sedgewick:ta:analy}) to obtain
  \begin{align*}
    \frac{q_b(m)}{((b-1)m+1)!}
%    &= [x^m] Q(x)
%    = \frac{1}{2\pi i} \oint_{\abs{x}=\,\textrm{small}} 
%    \frac{\f{Q}{x}}{x^{m+1}} \dd x \\
%    &= - \f{\operatorname{residue}}{\frac{\f{Q}{x}}{x^{m+1}}, x = \rho}
%   + \frac{1}{2\pi i} \oint_{\abs{x}=\rho+\epsilon} 
%   \frac{\f{Q}{x}}{x^{m+1}} \dd x \\
    &= - \frac{\f{S}{\rho}}{b!\f{T'}{\rho}} \rho^{-m}  
    + \Oh{\left(\frac{1}{\rho+\epsilon}\right)^m}.
  \end{align*}
  This finishes the proof for $s=1$. Note that $\gamma=1/\rho$.

  In the general case (arbitrary~$s$), we use the relation
  \begin{equation*}
    \sum_{n=1}^{\infty} \frac{\f{\W_b}{s,n}}{n!} x^n 
    = \bigg( \sum_{n=1}^{\infty} \frac{\f{\W_b}{1,n}}{n!} x^n \bigg)^s,
  \end{equation*}
  which follows from Equation~\eqref{eq:ws_equation} and gives us
  \begin{equation*}
    \sum_{m=0}^{\infty} \frac{\f{\W_b}{s,(b-1)m+s}}{((b-1)m+s)!} x^m 
    = Q(x)^s.
  \end{equation*}
Once again, we make use of the fact here that the (exponential) generating
  function is meromorphic, cf.\@ Section~\ref{sec:rec}. The singular expansion
  of $Q(x)^s$ at $x = \rho = 1/\gamma$ is given by
  \begin{equation*}
Q(x)^s = \Big( \frac{\alpha}{1-\gamma x} 
    + \Oh{1} \Big)^s,
  \end{equation*}
  which has $\alpha^s / (1-\gamma x)^s$ as its main term. Once again,
  singularity analysis~\cite{Flajolet-Odlyzko:1990:singul} yields the desired
  asymptotic formula with main term as indicated in the statement of the
  theorem.\end{proof}

% ------------------------------------------------------------------------------

\section{Reliable Numerical Calculations}
\label{sec:numerical}

% ------------------------------------------------------------------------------

We want to calculate the constants obtained in the previous sections in a
reliable way. The current section is devoted to this task. Our main tool will
be interval arithmetic, which is performed by the computer algebra system
Sage~\cite{Stein-others:2014:sage-mathem-6.3}.

% ------------------------------------------------------------------------------

For the calculations, we need bounds for the tails of our infinite sums. We
start with the following two remarks, which improves the bound found in
Section~\ref{sec:bounds}.

% ------------------------------------------------------------------------------

\begin{remark}\label{rem:bound-det-coeff}
  The bounds of Lemma~\ref{lem:bound-det-coeff} for the
  determinant~\eqref{eq:detexpression} can be tightened: 
%Instead of the
%  constant $c = (b-1)\left(\log\frac{b-1}{2\sqrt{2}}-1\right)$ we can use
%  $(b-1)\left(\log\frac{b-1}{\sqrt{2} + 1/\sqrt{n}}-1\right)$.
for an explicit $n$, we can calculate $\f{g}{n}$ more precisely by using the number of
  partitions of $n$ into distinct parts (and not a bound for that number) and
  similarly by using the factorial directly instead of Stirling's formula.

  An even better, but less explicit bound for the $n$th coefficient of
  $\det(I-\vc{M}(x))$ is given by
  \begin{equation}\label{eq:bound-det-coeff-specific-n}
    \abs{t_n} \leq \sum_{h \geq 0} h! \sum_{\substack{1\leq i_1<i_2<\dots<i_h \\ i_1,\dots,i_h\in\N \\ i_1+i_2+\cdots+i_h = n}}
\exp \Big( - (b-1)n \Big( \log 
        \frac{(b-1)n}{h} - 1 \Big) \Big).
  \end{equation}
  Note that we do not know whether this bound is decreasing in $n$ or
  not. However, for a specific $n$, one can calculate this bound, and it is
  much smaller than the general bounds above. For example, for $b=2$, we have
  $\abs{t_{60}} \leq 5.96\cdot 10^{-14}$ with this method, whereas
  Lemma~\ref{lem:bound-det-coeff} would give the bound $0.00014$.
\end{remark}

% ------------------------------------------------------------------------------

\begin{remark}\label{rem:bound-det-tail}
  We can also get tighter bounds in Lemma~\ref{lem:bound-det-tail}  using the
  ideas presented in Remark~\ref{rem:bound-det-coeff}. We can even use
  combinations of those bounds: For $M>N$, we separate
  \begin{equation*}
    \abs[Big]{\sum_{n \geq N} t_n x^n} 
    \leq \sum_{M > n \geq N} \abs{t_n} \abs{x}^n
    + \abs[Big]{\sum_{n \geq M} t_n x^n} 
  \end{equation*}
  and use the bound~\eqref{eq:bound-det-coeff-specific-n} for $M > n \geq N$
  and Lemma~\ref{lem:bound-det-tail} (tightened by some ideas from
  Remark~\ref{rem:bound-det-coeff}) for the sum over $n \geq M$. For example,
  again for $b=2$, we obtain the tail-bound
  \begin{equation*}
    \abs[Big]{\sum_{n \geq 60} t_n x^n} \leq 8.051\cdot 10^{-14} 
    + 4.068\cdot 10^{-15}
  \end{equation*}
  for $\abs{x} \leq 1$, where $M=86$ was chosen. (We will denote the constant
  on the right hand side of the inequality above by $B_{T_{60}}$, see the proof
  of Lemma~\ref{lem:zeros-det-b2}.) Using Lemma~\ref{lem:bound-det-tail}
  directly would just give $0.00103$.
\end{remark}

% ------------------------------------------------------------------------------

To get numerical values for the constants, we have to work with
\begin{equation*}
  Q(x) = 1 + \frac{x \f{S}{x}}{2\f{T}{x}},
\end{equation*}
where the first few terms of these power series are given by
\begin{equation*}
  \f{S}{x} = \vc{1}^T \operatorname{adj}(I-\vc{M}(x)) \vc{e}_1  = \det(M^*(x))
  = 1 - \tfrac{5}{12} x^2 - \tfrac{1}{6} x^3 
  - \tfrac{1}{24} x^4 + \tfrac{1}{45} x^5 + \cdots
\end{equation*}
and
\begin{equation*}
  \f{T}{x} = \det(I-\vc{M}(x))
  = 1 - x  - \tfrac{1}{2} x^2 + \tfrac{1}{6} x^3
  + \tfrac{1}{8} x^4 + \tfrac{3}{40} x^5 + \cdots,
\end{equation*}
cf.\@ Sections~\ref{sec:rec} and~\ref{sec:asymptotics}. We
obtain the following result for the denominator $\f{T}{x}$.

% ------------------------------------------------------------------------------

\begin{lemma}\label{lem:zeros-det-b2}
  For $b=2$, the function $\f{T}{x}$ has exactly one zero with $\abs{x} <
  \frac{3}{2}$. This simple zero lies at $x_0 = 0.83845184342\dots$.
\end{lemma}

% ------------------------------------------------------------------------------

\begin{remark}
  Note that $1/x_0 = \gamma = 1.192674341213\dots$, which is the constant found
  in Molteni~\cite{Molteni:2012:repr-2-powers-asy}.
\end{remark}

% ------------------------------------------------------------------------------

\begin{proof}[Proof of Lemma~\ref{lem:zeros-det-b2}]
  Denote the polynomials consisting of the first $N$ terms of $\f{T}{x}$ by
  $\f{T_N}{x}$. We have $\abs{\f{T}{x} - \f{T_{60}}{x}} \leq B_{T_{60}}$ with
  $B_{T_{60}} = 1.17\cdot 10^{-13}$, see Lemma~\ref{lem:bound-det-tail} and
  Remark~\ref{rem:bound-det-tail}. On the other hand, we have
  $\abs{\f{T_{60}}{x}} > 0.062$ for $\abs{x} = \frac32$ (the minimum is
  attained on the positive real axis) by using a bisection method together with
  interval arithmetic (in
  Sage~\cite{Stein-others:2014:sage-mathem-6.3}). Therefore, the functions
  $\f{T}{x}$ and $\f{T_{60}}{x}$ have the same number of zeros inside a disk
  $\abs{x}<\frac32$ by Rouch\'e's theorem ($0.062>B_{T_{60}}$). This number
  equals one, since there is only one zero, a simple zero, of $\f{T_{60}}{x}$
  with absolute value smaller than $\frac32$.

  To find the exact position of that zero consider $\f{T_{60}}{x} + B_{T_{60}}I$
  with the interval $I = [-1,1]$. Again, using a bisection method (starting with
  $\frac32 I$) plus interval arithmetic, we find an interval that contains
  $x_0$. From this, we can extract correct digits of $x_0$.
\end{proof}

% ------------------------------------------------------------------------------

From the previous result, we can calculate all the constants. The values of
those for the first few $b$ can be found in Table~\ref{tab:values-thm-asy}. The
following remark gives some details.

% ------------------------------------------------------------------------------

\begin{remark}
  As mentioned, to obtain reliable numerical values of all the constants
  involved in the statement of our theorems, we use the bounds obtained in
  Section~\ref{sec:bounds} together with interval arithmetic.

  Let $b=2$ and denote, as above, the polynomials consisting of the first $N$
  terms of $\f{S}{x}$ and $\f{T}{x}$, by $\f{S_N}{x}$ and $\f{T_N}{x}$
  respectively.  By the methods of Lemmas~\ref{lem:bound-det-coeff}
  and~\ref{lem:bound-det-tail} and Remarks~\ref{rem:bound-det-coeff}
  and~\ref{rem:bound-det-tail} we get, for instance, that $\abs{\f{T'}{x} -
    \f{T'_{60}}{x}} \leq B_{T'_{60}}$ with $B_{T'_{60}} = 8.397\cdot
  10^{-12}$. We also have $\abs{\f{S}{x} - \f{S_{60}}{x}} \leq B_{S_{60}}$ with
  $B_{S_{60}} = 1.848\cdot 10^{-13}$ for the function in the numerator of
  $\f{Q}{x}$. We plug $x_0$ into the approximations $S_{60}$ and $T'_{60}$ and
  use these bounds to obtain precise values (with guaranteed error estimates)
  for all the constants that occur in our formula.
\end{remark}

% ------------------------------------------------------------------------------

We finish this section with the following remark.

% ------------------------------------------------------------------------------

\begin{remark}
  If one does not insist on such explicit error bounds for the numerical
  approximations as above, one can get ``more precise'' numerical results (without formal proofs that all the digits are actually correct).
  Here, specifically, the first three terms in the asymptotic expansion are as
  follows: {\small
    \begin{align*}
      \f{\W_2}{1,n} / n! &= 0.296372049053529075588648642133 \cdot 
      1.192674341213466032221288982529^{n-1} \\
      &\mathrel{\phantom{=}} +\,0.119736335383631653495068554245 \cdot
      0.643427418149500070120570318509^{n-1} \\
      &\mathrel{\phantom{=}} +\,0.0174783635210388007051384381833 \cdot
      (-0.5183977738993377728627273570710)^{n-1} \\
      &\mathrel{\phantom{=}} +\,\cdots
    \end{align*}}
However, the numerical approximations lack the ``certifiability'' of
  e.g.\@ those in Table~\ref{tab:values-thm-asy}.
\end{remark}

% ------------------------------------------------------------------------------

\section{Maximum Number of Representations}
\label{sec:thm-max}

% ------------------------------------------------------------------------------

Let $\f{\U_b}{\ell,n}$ and $\f{\W_b}{s,n}$ be as defined in \eqref{eq:def-Ub}
in the introduction. In this section we analyze the function~$\f{M}{n} =
\f{M_b}{n}$, which equals the maximum of $\f{\U_b}{\ell,n}$ over all $\ell$,
i.e., we have
\begin{equation*}
  \f{M}{n} = \max_{\ell \geq 1}\, \f{\U_b}{\ell,n} 
  = \max_{s \geq 1}\, \f{\W_b}{s,n}.
\end{equation*}
This gives the maximum number of representations any positive integer can have
as the sum of exactly $n$ powers of $b$.

% ------------------------------------------------------------------------------

Throughout this section, we use the generating function
\begin{equation*}
  \f{W}{x} = \sum_{n=1}^{\infty} \frac{\f{\W_b}{1,n}}{n!} x^n.
\end{equation*}
Further, denote by $\theta$ the unique positive real solution (the power series
$W$ has real, nonnegative coefficients) with $\f{W}{\theta}=1$, and we set
$\nu=1/\theta$. We prove the following theorem, which is a generalized version
of Theorem~\ref{thm:maximum}.

% ------------------------------------------------------------------------------

\begin{thm}\label{thm:maximum-full}
  With the notions of $\f{W}{x}$, $\theta$ and $\nu$ as above, we have
  \begin{equation}\label{eq:M-upper-bound}
    \f{M}{n}/n! \leq \nu^n
  \end{equation}
  for all $n \geq 1$, and the constant is optimal: We have the more precise
  asymptotic formula
  \begin{equation*}
    \f{M}{n}/n! = \lambda n^{-1/2} \nu^n \big( 1 + \Oh[big]{n^{-1/2}} \big)
  \end{equation*}
  with $\lambda = (b-1) \left( \theta \f{W'}{\theta} \sigma \sqrt{2\pi}
  \right)^{-1}$,
  where $\sigma>0$ is defined by
  \begin{equation*}
    \sigma^2
    = \frac{\f{W''}{\theta}}{\theta \f{W'}{\theta}^3} 
    + \frac{1}{\theta^2 \f{W'}{\theta}^2} - \frac{1}{\theta \f{W'}{\theta}}.
  \end{equation*}
  Moreover, the maximum $\f{M}{n} = \max_{s\geq1}\, \f{\W_b}{s,n}$ is attained
  at $s = \mu n + \Oh{1}$
  with the constant $\mu = \left( \theta \f{W'}{\theta} \right)^{-1}$.
\end{thm}

% ------------------------------------------------------------------------------

In Table~\ref{tab:values-thm-max}, we are listing numerical values for the
constants of Theorem~\ref{thm:maximum-full}. These values are simply calculated by
using a finite approximation to $\f{W}{x}$, namely $\f{W_N}{x} = \sum_{n=1}^{N}
\frac{\f{\W_b}{1,n}}{n!} x^n$ for some precision~$N$.

% ------------------------------------------------------------------------------

\begin{table}
  \centering
  \begin{equation*}
    \begin{array}{cccccc}
      b & \lambda & \theta & \nu=1/\theta  & \mu & \sigma^2 \\
      \hline
      2 & 0.27693430 & 0.57071698 & 1.75218196 & 0.44867215 & 0.41775807 \\
3 & 0.70656285 & 0.84340237 & 1.18567368 & 0.66924459 & 0.57114748 \\
4 & 1.70314663 & 0.95872521 & 1.04305174 & 0.87318716 & 0.37650717 \\
5 & 4.20099030 & 0.99167231 & 1.00839763 & 0.96645454 & 0.13477198 \\
6 & 10.61691472 & 0.99861115 & 1.00139078 & 0.99304650 & 0.03480989 \\
7 & 28.28286119 & 0.99980159 & 1.00019845 & 0.99880929 & 0.00714564 \\
8 & 80.09108610 & 0.99997520 & 1.00002480 & 0.99982638 & 0.00121534 \\

      \hline
    \end{array}
  \end{equation*}
  \caption{Values (numerical approximations) for the constants of 
    Theorem~\ref{thm:maximum-full}. In the calculations the approximation
    $\f{W_{60}}{x}$ was used.}
  \label{tab:values-thm-max}
\end{table}

% ------------------------------------------------------------------------------

We start with the upper bound~\eqref{eq:M-upper-bound} of
Theorem~\ref{thm:maximum-full}, which is done in the following lemma.

% ------------------------------------------------------------------------------

\begin{lemma}\label{lem:M-upper-bound}
 We have
 \begin{equation*}
   \f{M}{n}/n! \leq \nu^n
 \end{equation*}
 for all $n \geq 1$.
\end{lemma}

% ------------------------------------------------------------------------------

\begin{proof}
  Recall that Equation~\eqref{eq:ws_equation} gives us
  \begin{equation*}
    \sum_{n=1}^{\infty} \frac{\f{\W_b}{s,n}}{n!} x^n 
    = \bigg( \sum_{n=1}^{\infty} \frac{\f{\W_b}{1,n}}{n!} x^n \bigg)^s 
    = \f{W}{x}^s.
  \end{equation*}
  Since $\theta > 0$ was chosen such that $\f{W}{\theta} = 1$, it clearly
  follows that
  \begin{equation*}
    \sum_{n=1}^{\infty} \frac{\f{\W_b}{s,n}}{n!} \theta^n = 1,
  \end{equation*}
  hence $\f{\W_b}{s,n}/n! \leq \theta^{-n}$ for all $s$ and $n$, and taking the
  maximum over all $s \geq 1$ yields
  \begin{equation*}
    \f{M}{n}/n! = \max_{s \geq 1}\, \f{\W_b}{s,n}/n! \leq \theta^{-n} = \nu^n,
  \end{equation*}
  which is what we wanted to show.
\end{proof}

% ------------------------------------------------------------------------------

It remains to prove the asymptotic formula for $\f{M}{n}$. We first
gather some properties of the solution $x=\f{\theta}{u}$ of the functional
equation $\f{W}{x}=1/u$.

% ------------------------------------------------------------------------------

\begin{lemma}\label{lem:alpha-minimal}
  For $u\in\C$ with $\abs{u}\leq1$ and $\abs{\Arg u} \leq \frac{\pi}{b-1}$, each
  root $x$ of $\f{W}{x} = 1/u$ satisfies the inequality $\abs{x} \geq \theta$,
  where equality holds only if $x = \theta$ and $u =1$.
\end{lemma}

% ------------------------------------------------------------------------------

\begin{proof}
  Let $u$ be as stated in the lemma. By the nonnegativity of the coefficients
  of $W$ and the triangle inequality, we have
  \begin{equation}\label{eq:alpha-minimal:ineq}
    \f{W}{\theta} = 1 \leq \abs{1/u} 
    = \abs{\f{W}{x}} \leq \f{W}{\abs{x}}.
  \end{equation}
  The first part of the lemma follows, since $W$ is increasing on the positive
  real line. It remains to determine when equality holds, so we assume in the
  following that $\abs{x} = \theta$.

  Since the coefficients $\f{\W_b}{1,n}$ are nonzero only for $n\equiv1 \mod
  b-1$, we can write $\f{W}{x} = x
  \f{V}{x^{b-1}}$. From~\eqref{eq:alpha-minimal:ineq}, we obtain
  \begin{equation*}
    W(\theta) = \theta \f[big]{V}{\theta^{b-1}} 
    = \abs{x} \abs[big]{\f[big]{V}{x^{b-1}}} = \abs{W(x)}.
  \end{equation*}
  Since the coefficients of $V$ are indeed positive, the power series $V$ is
  aperiodic\footnote{A power series is \emph{aperiodic} if the exponents whose
    associated coefficients are not zero are not contained in $a+b\Z$ for any
    $a$, $b$ with $b\geq2$.}. Therefore, the inequality
  $\abs{\f[big]{V}{x^{b-1}}} \leq \f[big]{V}{\abs[big]{x^{b-1}}}$ is strict,
  i.e., we have $\f[big]{V}{x^{b-1}} < \f[big]{V}{\abs[big]{x^{b-1}}}$ (which
  would yield a contradiction to the assumption that $\abs{x}=\theta$) unless
  $x^{b-1}$ is real and positive, which means that $x^{b-1} =
  \theta^{b-1}$. When this is the case, we have
  \begin{equation*}
    \frac{\theta}{u} = \theta \f{W}{x} = \theta x \f[big]{V}{x^{b-1}}
    = x \theta \f[big]{V}{\theta^{b-1}} = x \f{W}{\theta} = x,
  \end{equation*}
  so $\abs{\Arg x} = \abs{-\Arg u} \leq \frac{\pi}{b-1}$. This means that
  $x^{b-1}$ can only be real and positive if $x$ is itself real and positive,
  which implies that $x = \theta$ and $u = 1$.
\end{proof}

% ------------------------------------------------------------------------------

The following lemma tells us that the single dominant root of $\f{W}{x}=1$ is
the simple zero~$\theta$.

% ------------------------------------------------------------------------------

\begin{lemma}\label{lem:alpha-dominant-root}
  There exists exactly one root of $\f{W}{x}=1$ with $\abs{x}\leq\theta$,
  namely $\theta$. Further, $\theta$ is a simple root, and there exists an $\epsilon>0$ such that $\theta$ is the
  only root of $\f{W}{x}=1$ with absolute value less than
  $\theta+\epsilon$.
\end{lemma}

% ------------------------------------------------------------------------------

\begin{proof}
  By Lemma~\ref{lem:alpha-minimal} with $u=1$, the positive real $\theta$ is
  the unique root of $\f{W}{x}=1$ with minimal absolute value. This proves the
  first part of the lemma.

  Using Theorem~\ref{thm:asymptotics:general_base}, we get
  \begin{equation*}
    \abs{\f{W}{x}} = \f{O}{\sum_{m=0}^\infty \gamma^m \abs{x}^{(b-1)m}},
  \end{equation*}
  which is bounded for $\abs{x}<1/\gamma^{1/(b-1)}$. Therefore, the radius of
  convergence $r$ of $W$ is at least $1/\gamma^{1/(b-1)}>\theta$, and so $W$ is
holomorphic inside a circle that contains $\theta$. Since zeros of holomorphic
  functions do not accumulate, the existence of a suitable $\epsilon>0$ as desired
  follows.

  The root $\theta$ is simple, since $W(x)$ is strictly increasing on $(0,r)$.
\end{proof}

%% ------------------------------------------------------------------------------
%
%\begin{lemma}\label{lem:alpha-unique-min}
%  For $u\in\C$ let $\abs{\f{\theta}{u}} = \abs{x}$, where $x$ is a root of
%  $\f{W}{x}=1/u$ with smallest absolute value. Then, for $\varphi\in[-\pi,\pi]$,
%  the function $\abs{\theta(e^{i\varphi})}$ attains a unique minimum at
%  $\varphi=0$.
%\end{lemma}
%
%% ------------------------------------------------------------------------------
%
%\begin{proof}
%  Set $u=e^{i\varphi}$ and use Lemma~\ref{lem:alpha-minimal}: we obtain
%  $\abs{\f{\theta}{u}}\geq\theta$. If $\varphi\neq0$, then $\f{\theta}{u}$ is
%  not a positive real number, therefore we have $\abs{\f{\theta}{u}}>\theta$.
%\end{proof}

% ------------------------------------------------------------------------------

We are now ready to prove the asymptotic formula for $\f{M}{n}$. To this end,
we consider the bivariate generating function
\begin{equation*}
  \f{G}{x,u} =
  1 + \sum_{n=1}^{\infty} \sum_{s=1}^{\infty} \frac{\f{\W_b}{s,n}}{n!} x^nu^s 
  = \sum_{s=0}^{\infty} \f{W}{x}^su^s 
  = \frac{1}{1-u\f{W}{x}}.
\end{equation*}
In order to get $\max_{s\geq1}\, \f{\W_b}{s,n}$, we show that the
coefficients varying with $s$ fulfil a local limit law (as $n$ tends to
$\infty$). The maximum is then attained close to the mean.

% ------------------------------------------------------------------------------

\begin{proof}[Proof of Theorem~\ref{thm:maximum-full}]
  Set
  \begin{equation*}
    \f{g_n}{u} = [x^n] \f{G}{x,u} = \sum_{s=1}^{\infty} \frac{\f{\W_b}{s,n}}{n!} u^s.
  \end{equation*}
  We extract $g_n$ from the bivariate generating function $\f{G}{x,u}$. In
  order to do so, we proceed as in Theorem~IX.9 (singularity perturbation for
  meromorphic functions) of Flajolet and
  Sedgewick~\cite{Flajolet-Sedgewick:ta:analy}. An important detail here is the fact that $\frac{\f{\W_b}{s,n}}{n!} [u^sx^n] \f{G}{x,u}$ can only be nonzero if $s \equiv n \bmod (b-1)$, hence $g_n$ can also be expressed as
$$g_n(u) = u^r h_n(u^{b-1}),$$
where $r \in \{0,1,\ldots,b-2\}$ is chosen in such a way that $r \equiv n \bmod (b-1)$. This is also the reason why it was enough in Lemma~\ref{lem:alpha-minimal} to consider the case $|\Arg u| \leq \frac{\pi}{b-1}$. 

Now we check that all requirements for applying the quasi-power theorem are fulfilled.
By Lemma~\ref{lem:alpha-dominant-root}, the function $\f{G}{x,1}$ has a dominant simple pole at $x=\theta$ and no other singularities with absolute values smaller than $\theta+\epsilon$. The denominator $1-u \f{W}{x}$ is analytic and not degenerated at  $(x,u)=(\theta,1)$; the latter since its derivative with respect to $x$ is
$\f{W'}{\theta}\neq0$ ($\theta$ is a simple root of $F$) and its derivative
  with respect to $u$ is $-\f{W}{\theta}=-1\neq0$. 

Thus the function $\f{\theta}{u}$ which gives the solution to the equation $W(\theta(u)) = \theta$ with smallest modulus has the following properties: it is
  analytic at $u=1$, it fulfils $\theta(1)=\theta$, and for some $\epsilon > 0$ and $u$ in a suitable neighbourhood of $1$, there is no $x \neq \theta(u)$ with $W(x) =
  1/u$ and $|x| \leq \theta + \epsilon$.

  Therefore, by Cauchy's integral formula and the residue theorem, we obtain
  \begin{align*}
    \f{g_n}{u} &= -\f{\operatorname{Res}}{\frac{1}{1-u \f{W}{x}} x^{-n-1}, 
      x=\f{\theta}{u}} 
    +\frac{1}{2\pi i} \oint_{\abs{x}=\theta+\epsilon} \f{G}{z,u}\frac{\dd z}{z^{n+1}} \\
    &= \frac{1}{u \f{\theta}{u} \f{W'}{\f{\theta}{u}}} 
    \left(\frac{1}{\f{\theta}{u}}\right)^n
    + \f{O}{(\theta+\epsilon)^{-n}}
  \end{align*}
  for $u$ in a suitable neighbourhood of $1$. 

  To get the results claimed in Theorem~\ref{thm:maximum-full}, we use a
  local version of the quasi-power theorem, see Theorem IX.14
  of~\cite{Flajolet-Sedgewick:ta:analy} or Hwang's original
  paper~\cite{Hwang:1998:LLT}.  Set
  \begin{equation*}
    \f{A}{u} = \left(u \f{\theta}{u} \f{W'}{\f{\theta}{u}} \right)^{-1}
  \end{equation*}
  and
  \begin{equation*}
    \f{B}{u} = \left( \f{\theta}{u} \right)^{-1},
  \end{equation*}
  so that
  \begin{equation*}
    g_n(u) = A(u) B(u)^n + \Oh{(\theta+\epsilon)^{-n}}.
  \end{equation*}
  In terms of $h_n$, this becomes
  \begin{equation*}
    h_n(v) = v^{-r/(b-1)} A(v^{1/(b-1)}) B(v^{1/(b-1)})^n + \Oh{(\theta+\epsilon)^{-n}}.
  \end{equation*}
  Here, $v^{1/(b-1)}$ is taken to be the principal $(b-1)$th root of $v$, which
  satisfies $\abs{\Arg v^{1/(b-1)}} \leq \frac{\pi}{b-1}$.

  Since $\f{\theta}{u} \neq 0$ for $u$ in a suitable neighbourhood of $0$, the
  function $B$ is analytic at zero, and so is the function $A$ (since $W$ is
  analytic in a neighbourhood of $\theta(1) = \theta$ as well and has a nonzero
  derivative there). Moreover, we can use the fact that
  $\abs{\f{\theta}{e^{i\varphi}}}$ has a unique minimum at $\varphi=0$ if we
  assume that $\abs{\varphi} \leq \frac{\pi}{b-1}$ (which follows from
  Lemma~\ref{lem:alpha-minimal}).
  
  As a result, Theorem IX.14 of~\cite{Flajolet-Sedgewick:ta:analy} (slightly
  adapted to account for the periodicity of $g_n$) gives us
  \begin{equation}\label{eq:max-llt}
    \begin{split}
      \frac{\f{\W_b}{s,n}}{n!} &= 
      \frac{(b-1)\f{A}{1} \f{B}{1}^n}{\sigma \sqrt{2\pi n}}
      \f{\operatorname{exp}}{-\frac{z^2}{2\sigma^2}}
      \left(1+ \f{O}{\frac{1}{\sqrt{n}}}\right) \\
      &= \frac{(b-1)\nu^n}{\theta \f{W'}{\theta} \sigma \sqrt{2\pi n}}
      \f{\operatorname{exp}}{-\frac{z^2}{2\sigma^2}} \left(1+
        \f{O}{\frac{1}{\sqrt{n}}}\right),
    \end{split}
  \end{equation}
  where $z=(s-\mu n)/\sqrt{n}$. Mean and variance can be calculated as follows.
  We have
  \begin{equation*}
    \mu = \frac{\f{B'}{1}}{\f{B}{1}} 
    = - \frac{\f{\theta'}{1}}{\f{\theta}{1}} 
    = \frac{1}{\theta \f{W'}{\theta}},
  \end{equation*}
  and $\sigma>0$ is determined by
  \begin{align*}
    \sigma^2 &= \frac{\f{B''}{1}}{\f{B}{1}} 
    + \frac{\f{B'}{1}}{\f{B}{1}}
    - \left(\frac{\f{B'}{1}}{\f{B}{1}}\right)^2
    = - \frac{\f{\theta''}{1}}{\f{\theta}{1}} 
    - \frac{\f{\theta'}{1}}{\f{\theta}{1}}
    + \left(\frac{\f{\theta'}{1}}{\f{\theta}{1}}\right)^2 \\
    &= \frac{\f{W''}{\theta}}{\theta \f{W'}{\theta}^3} 
    - \frac{1}{\theta \f{W'}{\theta}} + \frac{1}{\theta^2 \f{W'}{\theta}^2},
  \end{align*}
  where we used implicit differentiation of $\f{W}{\f{\theta}{u}} = 1/u$ to get
  expressions for $\f{\theta'}{u}$ and $\f{\theta''}{u}$.

  The value $\f{\W_b}{s,n}/n!$ is maximal with respect to $s$ when $s = \mu n +
  \Oh{1}$. Its asymptotic value can then be calculated by~\eqref{eq:max-llt}.
\end{proof}

% ------------------------------------------------------------------------------

\section{The Largest Denominator and 
  the Number of Distinct Parts}
\label{sec:parameters}

% ------------------------------------------------------------------------------

In this last section we analyze some parameters of our compositions of~$1$. In particular, we will see that the exponent of the
largest denominator occurring in a random composition into a given number of powers of $b$ and
the number of distinct summands are both asymptotically normally distributed
and that their means and variances are of linear order.

Let us start with the largest denominator, for which we obtain the following theorem. Note
that we suppress the dependence on $b$ in all constants again.

% ------------------------------------------------------------------------------

\begin{thm}\label{thm:largest}
  The exponent of the largest denominator in a random composition of $1$ into
  $m = (b-1)n+1$ powers of $b$ is asymptotically normally distributed with mean
  $\mu_\ell n + \Oh{1}$ and variance $\sigma_\ell^2 n + \Oh{1}$.
\end{thm}

% ------------------------------------------------------------------------------

Numerical approximations to the values of $\mu_\ell$ and $\sigma_\ell^2$ can be
found in Table~\ref{tab:parameters}. The proof runs along the same lines as the proofs of Theorems~\ref{thm:asymptotics:general_base} and~\ref{thm:maximum-full}, so we only give a sketch here.

% ------------------------------------------------------------------------------

\begin{table}
  \centering
  \begin{equation*}
    \begin{array}{ccccc}
      b & \mu_\ell & \sigma_\ell^2 & \mu_d & \sigma_d^2  \\
      \hline
      2 & 0.81885148 & 2.38703164 & 0.71440975 & 2.13397882 \\
3 & 0.93352696 & 0.53468588 & 0.93318787 & 0.53600822 \\
4 & 0.97869416 & 0.15390515 & 0.97869416 & 0.15390519 \\
5 & 0.99366804 & 0.04335760 & 0.99366804 & 0.04335760 \\
6 & 0.99819803 & 0.01180985 & 0.99819803 & 0.01180985 \\
7 & 0.99950066 & 0.00315597 & 0.99950066 & 0.00315597 \\
8 & 0.99986404 & 0.00083471 & 0.99986404 & 0.00083471 \\

      \hline
    \end{array}
  \end{equation*}
  \caption{Values (numerical approximations) for the constants of 
    Theorems~\ref{thm:largest} and~\ref{thm:distinct}. In the numerical calculations 
    the power series were approximated by a polynomial consisting 
    of $40$ terms.}
  \label{tab:parameters}
\end{table}

% ------------------------------------------------------------------------------

\begin{proof}[Sketch of proof of Theorem~\ref{thm:largest}]
  We start by considering a bivariate generating function for the investigated
  parameter. In the recursive step described in Section~\ref{sec:rec} that led
  us to the identity~\eqref{eq:genfun_id}, the exponent of the largest denominator
  increases by $1$. Thus it is very easy to incorporate this parameter into the
  generating function. Indeed, if $\la(\vc{k})$ denotes the exponent of the
  largest denominator that occurs in a composition (or partition) $\vc{k}$,
  then the bivariate generating function
  \begin{equation*}
    \f{L_r}{x,y} = \sum_{n \geq 0} \sum_{\vc{k} \in C_{n,r}} 
    \frac{1}{((b-1)n+1)!}\, x^ny^{\la(\vc{k})}
    = \sum_{n \geq 0} \sum_{\vc{k} \in \Pa_{n,r}} \wt(\vc{k})\, x^n y^{\la(\vc{k})}
  \end{equation*}
  satisfies $\f{L_1}{x,y} = 1$ and
  \begin{equation*}
    \f{L_{bs}}{x,y} = x^s y \sum_{r \geq s} \frac{r!}{(r-s)!\,(bs)!} \f{L_r}{x,y}.
  \end{equation*}
  So if we set $\vc{V}(x,y) = (\f{L_b}{x,y}, \f{L_{2b}}{x,y}, \f{L_{3b}}{x,y},
  \ldots)^T$, then we now have
  \begin{equation*}
    \vc{V}(x,y) = \frac{xy}{b!} (I-y\,\vc{M}(x))^{-1} \vc{e}_1
  \end{equation*}
  in analogy to~\eqref{eq:vexplicit} with the same infinite matrix as in
  Section~\ref{sec:rec}. Moreover, we obtain
  \begin{equation*}
    \f{L}{x,y} = \sum_{r \geq 1} \f{L_r}{x} 
    = y + \vc{1}^T \vc{V}(x,y) 
    = 1 + \frac{xy}{b!} \vc{1}^T (I-y\,\vc{M}(x))^{-1} \vc{e}_1.
  \end{equation*}
It follows by the same estimates as in Section~\ref{sec:bounds} that this is a meromorphic function in $x$ for $y$ in a suitable neighbourhood of $1$. Thus our bivariate generating function belongs to the meromorphic scheme as described in Section IX.6 of~\cite{Flajolet-Sedgewick:ta:analy}, and the asymptotics of mean and variance are obtained by standard tools of singularity analysis. Asymptotic normality follows by Hwang's quasi-power theorem~\cite{Hwang:1998}.
\end{proof}

% ------------------------------------------------------------------------------

For the number of distinct parts we prove the following result.

% ------------------------------------------------------------------------------

\begin{thm}\label{thm:distinct}
  The number of distinct parts in a random composition of $1$ into $m =
  (b-1)n+1$ parts is asymptotically normally distributed with mean $\mu_d n +
  \Oh{1}$ and variance $\sigma_d^2 n + \Oh{1}$.
\end{thm}

% ------------------------------------------------------------------------------

Approximations of the constants can be found in
Table~\ref{tab:parameters}. Again we only sketch the proof, since it uses the same ideas.

% ------------------------------------------------------------------------------
\begin{proof}[Sketch of proof of Theorem~\ref{thm:distinct}]
Again, we consider a bivariate generating function. In the recursive step, the number
  of distinct parts increases by $1$, unless all fractions with highest
  denominator are split. In this case, the number of distinct parts stays the
  same. One can easily translate this to the world of generating functions: let
  $d(\vc{k})$ be the number of distinct parts in $\vc{k}$, and let
  $\f{D_r}{x,y}$ be the bivariate generating function, where $y$ now marks the
  number of distinct parts, i.e., we use
  \begin{equation*}
    \f{D_r}{x,y} = \sum_{n \geq 0} \sum_{\vc{k} \in C_{n,r}} 
    \frac{1}{((b-1)n+1)!}\, x^ny^{\di(\vc{k})}
    = \sum_{n \geq 0} \sum_{\vc{k} \in \Pa_{n,r}} \wt(\vc{k})\, x^n y^{\di(\vc{k})}.
  \end{equation*}
  Then we have $\f{D_1}{x,y} = y$ and
  \begin{equation*}
    \f{D_{bs}}{x,y} = \frac{s!\, x^s}{(bs)!} \f{D_s}{x,y}
    + x^s y \sum_{r > s} \frac{r!}{(r-s)!\,(bs)!} \f{D_r}{x,y}.
  \end{equation*}
  Once again, we take the infinite vector $\vc{V}(x,y) = (\f{D_b}{x,y},
  \f{D_{2b}}{x,y}, \f{D_{3b}}{x,y}, \ldots)^T$, and we define a modified
  version $\vc{M}^*$ of the infinite matrix by its entries
  \begin{equation*}
    m_{ij}^* = 
    \begin{cases}
      \frac{(bj)!\,x^iy}{(bj-i)!\,(bi)!} & \text{if } i < bj, \\
      \frac{i!\,x^i}{(bi)!} & \text{if } i = bj, \\
      0 & \text{otherwise.}
    \end{cases}
  \end{equation*}
  Now
  \begin{equation*}
    \vc{V}(x,y) = \frac{xy}{b!} (I-\vc{M}^*(x))^{-1} \vc{e}_1
  \end{equation*}
  in analogy to~\eqref{eq:vexplicit}, and moreover
  \begin{equation*}
    \f{D}{x} = \sum_{r \geq 1} \f{D_r}{x}
    = y + \vc{1}^T \vc{V}(x,y)
    = y + \frac{xy}{b!} \vc{1}^T (I-\vc{M}^*(x))^{-1} \vc{e}_1.
  \end{equation*}
Once again, we find that the bivariate function belongs to the meromorphic scheme, so that we can apply singularity analysis and the quasi-power theorem to obtain the desired result.
\end{proof}

% ------------------------------------------------------------------------------

\renewcommand{\MR}[1]{}
\bibliographystyle{amsplain}
\bibliography{compositions-full}

% ------------------------------------------------------------------------------

\end{document}